\documentclass[12pt, reqno]{amsart}

\usepackage{amsmath, amssymb, amsthm}
\usepackage[pdftex,letterpaper=true,colorlinks=true,pdfpagemode=none,urlcolor=blue,linkcolor=blue,citecolor=blue,pdfstartview=FitH]{hyperref}

\newtheorem{theorem}{Theorem}
\newtheorem{lemma}{Lemma}
\newtheorem{proposition}{Proposition}

\newtheorem{corollary}{Corollary}

\newtheorem{remark}{Remark}

\numberwithin{equation}{section}

\begin{document}

\title[Weighted isoperimetric inequalities]{Weighted isoperimetric inequalities in warped product manifolds}

\author{Kwok-Kun Kwong}
\address{Department of Mathematics, National Cheng Kung University, Tainan City 701, Taiwan}
\email{kwong@mail.ncku.edu.tw}

\thanks{Research partially supported by Ministry of Science and Technology in Taiwan under grant MOST103-2115-M-006-016-MY3.}


\begin{abstract}
We prove some isoperimetric type inequalities in warped product manifolds, or more generally, multiply warped product manifolds. We then relate them to inequalities involving the higher order mean-curvature integrals. We also apply our results to obtain sharp eigenvalue estimates and some sharp geometric inequalities in space forms.
\end{abstract}
\maketitle
\section{Introduction}\label{sec: intro}

The classical isoperimetric inequality on the plane states that for a simple closed curve on $\mathbb R^2$, we have $L^2\ge 4\pi A$, where $L$ is the length of the curve and $A$ is the area of the region enclosed by it. The equality holds if and only if the curve is a circle. The classical isoperimetric inequality has been generalized to hypersurfaces in higher dimensional Euclidean space, and to various ambient spaces. For these generalizations, we refer to the beautiful article by Osserman \cite{osserman1978isoperimetric} and the reference therein. For a more modern account see \cite{ros2001isoperimetric}. Apart from two-dimensional manifolds and the standard space forms $\mathbb R^n$, $\mathbb H^n$ and $\mathbb S^n$, there are few manifolds for which the isoperimetric surfaces are known. According to \cite{bray2002isoperimetric}, known examples include $\mathbb R\times \mathbb H^n$, $\mathbb RP^3$, $\mathbb S^1\times \mathbb R^2$, $T^2\times \mathbb R$, $\mathbb R\times \mathbb S^n$, $\mathbb S^1\times \mathbb R^n$, $\mathbb S^1\times \mathbb S^2$, $\mathbb S^1\times \mathbb H^2$ and the Schwarzschild manifold, most of which are warped product manifolds over an interval or a circle. There are also many applications of the isoperimetric inequalities. For example, isoperimetric surfaces were used to prove the Penrose inequality \cite{bray2001proof}, an inequality concerning the mass of black holes in general relativity, in some important cases.

In this paper, we prove both classical and weighted isoperimetric results in warped product manifolds, or more generally, multiply warped product manifolds. We also relate them to inequalities involving the higher order mean-curvature integrals. Some applications to geometric inequalities and eigenvalues are also given.

For the sake of simplicity, let us describe our main results on a warped product manifold. The multiply warped product case is only notationally more complicated and presents no additional conceptual difficulty.

Let $\displaystyle M=[0, \lambda)\times N$ ($\lambda\le \infty$) be a product manifold.
Equip $M$ with the warped product Riemannian metric $\displaystyle g=dr^2+ s(r)^2 g_{N}$ for some continuous $s(r)\ge 0$, where $g_{N}$ is a Riemannian metric on the $m$-dimensional manifold $N$,  which we assume to be compact and oriented.
Define $B_R:=\{(r, \theta)\in M: r< R\}$.
We define the functions $A(r)$ and $v(r)$ by
\begin{align*}
A(r):= s(r)^m
\; \textrm{ and }\;
v(r):= \int_{0}^{r} A(t) dt.
\end{align*}
Up to multiplicative constants, they are just the area of $\partial B_r$ and the volume of $B_r$ respectively. For a bounded domain $\Omega$ in $M$, we define $\Omega^\#$ to be the region $B_R$ which has the same volume as $\Omega$, i.e. $\mathrm{Vol} (B_R) =\mathrm{Vol} (\Omega) $. We denote the area of $\partial \Omega$ and $\partial \Omega^\#$ by $|\partial \Omega|$ and $|\partial \Omega^\#|$ respectively.

One of our main results is the following isoperimetric theorem, which is a special case of Theorem \ref{thm weighted vol}.
\begin{theorem}\label{thm1}
Let $\Omega$ be a bounded open set in $(M,g)$ with Lipschitz boundary.
Assume that
\begin{enumerate}
\item \label{cond: surj}
The projection map $\pi: \partial \Omega\subset \mathbb [0, \lambda)\times N\to N$ defined by $(r, \theta)\mapsto \theta$ is surjective.
\item\label{cond: mono}
$s(r)$ is non-decreasing.
\item\label{cond: convex}
$A\circ v^{-1}$ is convex, or equivalently, $ss''-s'^2\ge 0$ if $s$ is twice differentiable.
\end{enumerate}
Then the isoperimetric inequality holds: $$|\partial \Omega|\ge|\partial \Omega^\#|.$$
\end{theorem}

We remark that our notion of convexity does not require the function to be differentiable: $f$ is convex on $I$ if and only if $f((1-t)x+ty)\le (1-t)f(x)+tf(y)$ for any $t\in(0,1)$ and $x, y\in I$.

One feature of our result is that except compactness, we do not impose any condition on the base manifold $N$. We will see in Section \ref{sec: necess} that without further restriction on $N$, our conditions are optimal in a certain sense.

The expression $ss''-s'^2$ comes from the observation that
if $s$ is twice differentiable, then as $v'=s^m$, the convexity of
$A(v^{-1}(u))=s\left( v^{-1}\left( u \right)\right)^m$ is equivalent to
\begin{equation}\label{eq: conv}
\frac{d^2}{du^2} A \left(v^{-1}(u)\right)
=\frac{m}{ s(r)^{m+2}}\left(s(r)s''(r)-s'(r)^2\right)\ge 0,
\end{equation}
where $r=v^{-1}(u)$.

The expression $ss''-s'^2$ also has a number of geometric and physical meanings. It is related to the stability of the slice $\Sigma=\{r=r_0\}$ as a constant-mean-curvature (CMC) hypersurface (i.e. whether it is a minimizer of the area among nearby hypersurfaces enclosing the same volume). Indeed, it can be shown (\cite{GLW} Proposition 6.2) that $\lambda_1(g_N)\ge m\left(s(r_0)'^2-s(r_0)s''(r_0)\right)$ if and only if $\Sigma$ is a stable CMC hypersurface, where $\lambda_1(g_N)$ is the first Laplacian eigenvalue of $(N, g_N)$. It is also related to the so called ``photon spheres'' in relativity (see \cite{GLW} Proposition 6.1).

From \eqref{eq: conv}, $A\circ v^{-1}$ is convex if and only if $s$ is $\log$-convex. So if $f(r)$ is non-decreasing and convex, then $s(r)=\exp(f(r))$ satisfies the convexity condition. One example of such a function is $s(r)=e^r$.

If $\partial \Omega$ is star-shaped, we can remove the assumption on the monotonicity of $s(r)$.
\begin{theorem}\label{thm2}
Suppose $\Omega$ is a bounded open set in $(M,g)$ with Lipschitz boundary.
Assume that
\begin{enumerate}
\item
$\partial \Omega$ is star-shaped in the sense that it is a graph over $N$, i.e. of the form $\partial \Omega=\{(r, \theta): r=\psi(\theta), \theta\in N\}$.
\item\label{cond2 thm2}
$A\circ v^{-1}$ is convex,
or equivalently, $ss''-s'^2\ge 0$ if $s$ is twice differentiable.
\end{enumerate}

Then the isoperimetric inequality holds: $$|\partial \Omega|\ge|\partial \Omega^\#|.$$
\end{theorem}
If the classical isoperimetric inequality already holds on $(M,g)$, we can extend it by the following result, which is a special case of Theorem \ref{thm: main}.
\begin{theorem}[Weighted isoperimetric inequality]\label{thm3}
Let $\Omega$ be a bounded open set in $(M,g)$ with Lipschitz boundary.
Assume that
\begin{enumerate}
\item
The classical isoperimetric inequality holds on $(M,g)$, i.e. $|\partial \Omega|\ge|\partial \Omega^\#|$.
\item
$a(r)$ is a non-negative continuous function such that
$\psi(r):=b(r)A(r)$ is non-decreasing, where $b(r):=a(r)-a(0)$.
\item
The function
$\psi\circ v^{-1} $ is convex.
\end{enumerate}
Then
\begin{align*}
\int_{\partial \Omega} a(r)dS
\ge\int_{\partial \Omega^{\#}} a(r)dS.
\end{align*}
If $b(r)A(r)>0$ for $r>0$, then the equality holds if and only if $r=\mathrm{constant}$, i.e. $\partial \Omega$ is a coordinate slice.
\end{theorem}

Using a volume preserving flow, recently Guan, Li and Wang \cite{GLW} (see also \cite{guan2015mean}) proved the following related result, assuming $s$ is smooth ($\partial \Omega$ is called ``graphical'' in \cite{GLW}):
\begin{theorem} (\cite[Theorem 1.2]{GLW}) \label{thm: GLW}
Suppose $\Omega$ is a domain in $(M,g)$ with smooth star-shaped boundary. Assume that
\begin{enumerate}
\item
The Ricci curvature $\mathrm{Ric}_N$ of $g_N$ satisfies
$\mathrm{Ric}_N\ge (m-1)K g_N$, where $K>0$ is constant.
\item
$0\le s'^2-ss''\le K$.
\end{enumerate}
Then the isoperimetric inequality holds: $$|\partial \Omega|\ge|\partial \Omega^\#|.$$
If $ s'^2-ss''<K$, then the equality holds if and only if $\partial \Omega$ is a coordinate slice $\{r=r_0\}$.
\end{theorem}
We note that our assumption $ss''-s'^2\ge 0$ in Theorem \ref{thm2} \textit{complements} that of \cite{GLW}. This does not contradict the result in \cite{GLW}. In fact, we will show the necessity of this and other conditions in Section \ref{sec: necess} (cf. Proposition \ref{prop: necess}). We also notice that except the obvious case that $M$ has constant curvature, the equality holds only when $\partial \Omega$ is a coordinate slice.
Indeed, combining Theorem \ref{thm: GLW} with Theorem \ref{thm3}, we can generalize Theorem \ref{thm: GLW} as follows.
\begin{theorem}[Theorem \ref{thm: glw weighted}]
Suppose $\Omega$ is a domain in $(M,g)$ with smooth star-shaped boundary.
Assume that the Ricci curvature $\mathrm{Ric}_N$ of $g_N$ satisfies
$\mathrm{Ric}_N\ge (m-1)K g_N$ and $0\le s'^2-ss''\le K$, where $K>0$ is constant. Suppose $a(r)$ is a positive continuous function such that $b( v^{-1}(u))s( v^{-1}(u))^m$ is convex, where $b(r):=a(r)-a(0)$.
Then the weighted isoperimetric inequality holds: $$\int_{\partial \Omega}a(r)dS\ge \int_{\partial \Omega^\#}a(r)dS.$$
The equality holds if and only if either
\begin{enumerate}
\item
$(M,g)$ has constant curvature, $a(r)$ is constant on $\partial \Omega$, and $\partial \Omega$ is a geodesic hypersphere, or
\item
$\partial \Omega$ is a slice $\{r=r_0\}$.
\end{enumerate}
\end{theorem}
Combining Theorem \ref{thm: GLW}, Theorem \ref{thm2} and the proof of Proposition \ref{prop: necess}, we get the following general picture for the isoperimetric problem in warped product manifolds:
\begin{theorem}
Let $M$ be the product manifold $[0, \lambda)\times N$ equipped with the warped product metric $g=dr^2+s(r)^2 g_N$.
\begin{enumerate}
\item
Suppose $s'^2-ss''\le 0$. Then the star-shaped isoperimetric hypersurfaces are precisely the coordinate slices $\{r=r_0\}$.
\item
Suppose $0\le s'^2-ss''\le K$ and $\mathrm{Ric}_N\ge (m-1)Kg_N$ where $K>0$ is constant. Then the star-shaped isoperimetric hypersurfaces are either geodesic hyperspheres if $(M,g)$ has constant curvature, or the coordinate slices $\{r=r_0\}$.
\item
Suppose $s'^2-ss''>K$ and $\lambda_1(g_N)\le mK$ where $K>0$ is constant. Then the coordinate slices $\{r=r_0\}$ cannot be isoperimetric hypersurfaces.

\end{enumerate}

\end{theorem}

We also prove isoperimetric type theorems involving the integrals of higher order mean curvatures $(H_k)$ in warped product manifolds. For simplicity, let us state the result when the ambient space is $\mathbb R^n$ (Corollary \ref{cor: Rn}), which follows from a more general theorem (Theorem \ref{thm: higher})
\begin{theorem}[Corollary \ref{cor: Rn}]
Let $\Sigma$ be a closed embedded hypersurface in $\mathbb R^{m+1}$
which is star-shaped with respect to $0$ and $\Omega$ is the region enclosed by it. Assume that $H_k>0$ on $\Sigma$.
Then for any integer $l\ge 0$,
\begin{align*}
n \beta_n^{-\frac{l-1}{n}}\mathrm{Vol}(\Omega)^{\frac{n-1+l}{n}}
\le\int_{\Sigma} H_kr^{l+k}dS,
\end{align*}
where $\beta_n$ is the volume of the unit ball in $\mathbb R^n$.
If $l\ge 1$, the equality holds if and only if $\Sigma$ is a sphere centered at $0$.
\end{theorem}
Note that when $k=l=0$, this reduces to $n \beta_n^{\frac{1}{n}}\mathrm{Vol}(\Omega)^{\frac{n-1}{n}}
\le \mathrm{Area} (\Sigma) $, which is the classical isoperimetric inequality. In fact, we prove a stronger result \eqref{ineq: chain} :
\begin{align*}
n \beta_n^{-\frac{l-1}{n}}\mathrm{Vol}(\Omega)^{\frac{n-1+l}{n}}
\le \int_{\partial \Omega} H_0 r ^{l} dS
\le \int_{\partial \Omega} H_1 r ^{l+1} dS
\le \cdots
\le \int_{\partial \Omega} H_k r ^{l+k} dS.
\end{align*}
This can be compared to the following result of Guan-Li \cite[Theorem 2]{Guan-Li}:
\begin{align*}
\left(\frac{1}{ \beta_n }\mathrm{Vol}(\Omega)\right)^{\frac{1}{n}}\le \left(\frac{1}{n\beta_n}{\int_\Sigma H_0dS}\right)^{\frac{1}{n-1}}\le \cdots \le  \left(\frac{1}{n\beta_n} \int_\Sigma H_k dS \right) ^{\frac{1}{n-k-1}}
\end{align*}
under the same assumption.

Some applications of the weighted isoperimetric inequalities will also be given in Section \ref{sec: eg} and Section \ref{sec: eigen}.

The rest of this paper is organized as follows. In Section \ref{sec: main}, we first prove Theorem \ref{thm3}. In Section \ref{sec: weighted vol}, we prove the isoperimetric inequality involving a weighted volume (Theorem \ref{thm weighted vol}), which implies Theorem \ref{thm1} and Theorem \ref{thm2}. Although Theorem \ref{thm3} can also be stated using the weighted volume, we prefer to prove the version involving only the ordinary volume for the sake of clarity, and indicates the changes needed to prove Theorem \ref{thm weighted vol}. In Section \ref{sec: eg}, we illustrate how we can obtain interesting geometric inequalities in space forms by using Theorem \ref{thm: main}. In Section \ref{sec: higher}, we introduce the weighted Hsiung-Minkowski formulas in warped product manifolds, and combine them with the isoperimetric theorem to obtain new isoperimetric results involving the integrals of the higher order mean curvatures. In Section \ref{sec: eigen}, we give further applications of our results to obtain some sharp eigenvalue estimates for the Steklov differential operator (also known as Dirichlet-to-Neumann map) and a second order differential operator related to the extrinsic geometry of hypersurfaces. Finally in Section \ref{sec: necess}, we show that the conditions of Theorem \ref{thm1} are all necessary by giving counterexamples where the isoperimetric inequality fails if any one of the conditions is violated.

\section{Weighted isoperimetric inequality on multiply warped product manifolds}\label{sec: main}
In this section, we first prove Theorem \ref{thm3}.
As explained in Section \ref{sec: intro}, our result actually applies to multiply warped product manifold with no additional difficulties.  Let us now describe our setting.

Let $\displaystyle M=[0, \lambda)\times \prod_{q=1}^{p}N_q$ ($\lambda\le \infty$) be a product manifold.
Equip $M$ with the warped product Riemannian metric $\displaystyle g:=dr^2+ \sum_{q=1}^{p}s_q(r)^2 g_{N_q}$ for some $s_q(r)\ge 0$, where $g_{N_q}$ is a Riemannian metric on the $m_q$-dimensional manifold $N_q$.

Denote the $m$-dimensional volume of $\displaystyle N:=\prod_{q=1}^{p}N_q$ with respect to the product metric $\displaystyle \sum_{q=1}^{p}g_{N_q}$ by $|N|$ , where $\displaystyle m=\sum_{q=1}^{p}m_q$.
Define $B_R:=\{(r, \theta)\in M: r< R\}$.
We define the functions $A(r)$, $v(r)$ and $V(r)$ by

\begin{align*}
A(r):= \prod_{q=1}^{p}s_q(r)^{m_q},\;
v(r):= \int_{0}^{r} A(t) dt \;\textrm{ and }\;
V(r):= |B_r| =|N| v(r).
\end{align*}

Here $|B_R|$ denotes that $(m+1)$-dimensional volume of $B_R$ with respect to $g$. Note that $\psi\circ v^{-1}$ is convex if and only if $\psi\circ V^{-1}$ is convex.

For a bounded domain $\Omega$ in $M$, we define $\Omega^\#$ to be the region $B_R$ which has the same volume as $\Omega$, i.e. $|B_R|=|\Omega|$.

In this paper, we will assume that $A(r)>0$ and can possibly be zero only when $r=0$. However, if $\{r=0\}$ is identified as a point, we do not assume $g$ is smooth at this point, e.g. metric with a conical singularity at this point is allowed.

\begin{theorem}\label{thm: main}
Let $\Omega$ be a bounded open set in $(M,g)$ with Lipschitz boundary.
Assume that
\begin{enumerate}
\item\label{cond1}
The classical isoperimetric inequality holds on $(M,g)$, i.e. $|\partial \Omega|\ge|\partial \Omega^\#|$.
\item\label{cond2}
$a(r)$ is a non-negative continuous function such that
$b(r)A(r)$ is non-decreasing, where $b(r):=a(r)-a(0)$.
\item \label{cond3}
The function
$b (V^{-1}(u)) \;A (V^{-1}(u))$ is convex.
\end{enumerate}
Then
\begin{align*}
\int_{\partial \Omega} a(r)dS
\ge\int_{\partial \Omega^{\#}} a(r)dS.
\end{align*}
If $b(r)A(r)>0$ for $r>0$, then the equality holds if and only if $r=\mathrm{constant}$, i.e. $\partial \Omega$ is a coordinate slice.
\end{theorem}

One of the ingredients of the proof of Theorem \ref{thm: main}
is the Jensen's inequality. The use of Jensen's inequality in establishing isoperimetric results in $\mathbb R^n$ has also been implemented in \cite{betta1999weighted}.

Let us give some notations.
Suppose $\phi$ is a monotone function on $\mathbb R$ and $\mu$ is a probability measure on $X$ (i.e. $\mu(X)=1$), we then define for a function $\rho: X\to \mathbb R$
$$\mathcal M_{\phi,\mu}[\rho]:=\phi^{-1}\left(\int_X \phi(\rho)d\mu\right).$$
The following form of Jensen's inequality will be useful to us.
\begin{proposition} \label{prop: jensen}
Let $\mu$ be a probability measure on $X$ and $\phi, \psi$ be functions on $\mathbb R$. Assume $\phi^{-1}$ exists and $\psi\circ \phi^{-1}$ is convex, then
$$\psi \left(\mathcal M_{\phi, \mu}[\rho]\right)
\le\int_{X} \psi (\rho) d\mu. $$
Moreover, if $\psi$ is strictly increasing, then
$\mathcal M_{\phi, \mu}[\rho] \le \mathcal M_{\psi, \mu}[\rho]$.
\end{proposition}

\begin{proof}
Define $\Phi=\psi\circ \phi^{-1}$. Since $ {\Phi}$ is convex, by Jensen's inequality,
\begin{align*}
\psi \left(\mathcal M_{\phi, \mu}[\rho]\right)
=\Phi\left(\int_X \phi (\rho) d\mu\right)
\le \int_{X}\Phi (\phi (\rho)) d\mu
= \int_{X} \psi (\rho) d\mu.
\end{align*}
If $\psi$ is strictly increasing, applying $\psi^{-1}$ to the above inequality, we get $\mathcal M_{\phi, \mu}[\rho] \le \mathcal M_{\psi, \mu}[ \rho]$.
\end{proof}
We now prove Theorem \ref{thm: main}. We remark that the reader may feel free to assume $p=1$ without affecting one's understanding of the proof.
\begin{proof}
Assume first $\Sigma=\partial \Omega$ is piecewise $C^1$, and that $\Sigma$ is a union of graphs over finitely many domains in $N$.
This means that there exists open, pairwise disjoint subsets $\{S_i\}_{i=1}^l$ of $N$ with Lipschitz boundary such that
$\Sigma$ is represented by
\begin{align*}
\Sigma=\partial \Omega=\left\{(r, \theta): r=r_{i,j}(\theta), \theta\in \overline S_i, j\in\{1,\cdots, 2k_i\}, i\in\{1, \cdots, l\}\right\}
\end{align*}
and
\begin{align*}
\overline \Omega
=\left\{(r, \theta): r_{i, 2\kappa-1}(\theta)\le r\le r_{i, 2\kappa}(\theta), \theta\in \overline S_i, \kappa\in\{1,\cdots, k_i\}, i\in\{1,\cdots, l\}\right\},
\end{align*}
where
\begin{align*}
&r_{i,j}\in C^1(S_i)\cap C^0(\overline S_i),\quad j=1,\cdots, 2k_i, \\
&r_{i,1}(\theta)<\cdots< r_{i,2k_i}(\theta)\quad \textrm{for } \theta\in S_i, \\
&r_{i,1}(\theta)
\begin{cases}
=0\quad \textrm{if }(0, \theta)\in \Omega\\
>0\quad \textrm{if }(0, \theta)\notin \Omega.
\end{cases}
\end{align*}
Let $\displaystyle S=\bigcup_{i=1}^l S_i$,
then by direct computation,
\begin{align*}
\int_\Sigma a(r) dS
=&\int_{S} a(r)\prod_{q=1}^{p}\left(1+s_q(r)^{-2}|\nabla_{N_q} \, r|^2_{g_{N_q}}\right)^{\frac{1}{2}}s_q(r)^{m_q}d\mathrm{vol}_N.
\end{align*}
Here $\nabla_{N_q} $ is the connection with respect to $g_{N_q}$ and $d\mathrm{vol}_N$ is the $m$-dimensional volume form on $N$.
Let
\begin{equation*}
I:=\int_{\partial \Omega} b(r)dS\textrm{ and }I^{\#}:=\int_{\partial \Omega^{\#}} b(r)dS.
\end{equation*}
We claim that
\begin{equation}\label{ineq: claim}
I\ge I^\#.
\end{equation}
Let $\psi(r)=b(r)A(r)=b(r)\prod_{q=1}^{p}s_q(r)^{m_q}$. First of all,
\begin{equation}\label{ineq: I}
\begin{split}
I
=&\sum_{i=1}^{l}\sum_{j=1}^{2k_i}\int_{S_i}b(r_{i,j})\prod_{q=1}^{p}\left(1+s_q(r_{i,j})^{-2}\left|\nabla_{N_q}\,r_{i,j}\right|^2_{g_{N_q}}\right)^{\frac{1}{2}}s_q(r_{i,j})^{m_q}d\mathrm{vol}_N\\
\ge&\sum_{i=1}^{l}\sum_{j=1}^{2k_i}\int_{S_i}b(r_{i,j})\prod_{q=1}^{p}s_q(r_{i,j})^{m_q}d\mathrm{vol}_N\\
=&\sum_{i=1}^{l}\sum_{j=1}^{2k_i}\int_{S_i} \psi(r_{i,j})
d\mathrm{vol}_N\\
\ge&\sum_{i=1}^{l}\int_{S_i} \psi(r_{i, 2k_i}) d\mathrm{vol}_N\\
=&\int_{N}\psi(\rho) d\mathrm{vol}_N
\end{split}
\end{equation}
where we define
$\rho(\theta):=
\begin{cases}
r_{i,2k_i}(\theta)\; &\textrm{if }\theta\in S_i,\\
0 \;&\textrm{if }\theta\in N\setminus S
\end{cases}
$, noting that $b(0)=0$.

On the other hand, for $B_R=\Omega^{\#}$, we have
\begin{align*}
|B_R| = |N| \int_{0}^{R} A(r)dr=|N|v(R).
\end{align*}
As $|B_R|=|\Omega|$, it is not hard to see by Fubini's theorem that
\begin{equation}\label{eq: R}
|B_R| =|\Omega|=\sum_{i=1}^{l}\sum_{j=1}^{2k_i}(-1)^j\int_{S_i} v(r_{i,j}) d\mathrm{vol}_N.
\end{equation}
Define $R_1$ by
\begin{align*}
\sum_{i=1}^{l}\sum_{j=1}^{2k_i}(-1)^j\int_{S_i} v(r_{i,j}) d\mathrm{vol}_N
\le&\sum_{i=1}^{l}\int_{S_i} v(r_{i,2k_i}) d\mathrm{vol}_N=: |B_{R_1}|,
\end{align*}
noting that $v(r)$ is increasing.
Then by the definition of $V$ and $\rho$,
\begin{equation}\label{eq: R1}
R_1 = V^{-1} \left(\int_{N} v(\rho) d\mathrm{vol}_N\right).
\end{equation}
Comparing with \eqref{eq: R}, we have $R_1\ge R$, and as $\psi(r)$ is non-decreasing,
\begin{equation}\label{ineq: I sharp}
I^{\#}=|N|b(R)A(R)=|N|\psi(R)\le |N| \psi(R_1).
\end{equation}

Now take $d\mu=\frac{1}{|N|}d\mathrm{vol}_N$.
As $\psi\circ V^{-1}$ is convex, by \eqref{ineq: I}, Jensen's inequality (Proposition \ref{prop: jensen}), \eqref{eq: R1} and \eqref{ineq: I sharp},
\begin{equation}\label{ineq: pf1}
\begin{split}
\frac{1}{|N|}I
\ge\frac{1}{|N|}\int_{N}\psi(\rho) d\mathrm{vol}_N
=&\int_{N}\psi(\rho) d\mu\\
\ge& \psi \left(\mathcal M_{V, \mu}[\rho]\right)\\
=& \psi \left(V^{-1}\left( \int_N V (\rho) d\mu\right)\right)\\
=& \psi\left(V^{-1}\left(\int_{N} v(\rho)d\mathrm{vol}_N \right)\right)\\
=&\psi(R_1)\\
\ge&\frac{1}{|N|}I^\#.
\end{split}
\end{equation}

We have proved \eqref{ineq: claim}. Moreover, from \eqref{ineq: I},  if $\psi(r)>0$ for $r>0$, then $r_{i,j}$ must be locally constant and hence $r$ is constant on $\Sigma$.

Finally, by the classical isoperimetric inequality,
\begin{equation*}
\begin{split}
\int_{\partial \Omega} a(r)dS
=&I+ a(0)\int_{\partial \Omega} dS\\
\ge &I^{\#}+ a(0)\int_{\partial \Omega^{\#}} dS\\
=&\int_{\partial \Omega^{\#}} a(r)dS.
\end{split}
\end{equation*}
In general, for a domain $\Omega$ with Lipschitz boundary, we can approximate $\Omega$ by piecewise $C^1$ domains $\Omega_i$ which satisfy the above conditions. A standard approximation argument will then give the desired result.
\end{proof}

As explained in Section \ref{sec: intro}, combining Theorem \ref{thm: GLW} (\cite[Theorem 1.2]{GLW}) with Theorem \ref{thm: main} we have
\begin{theorem}\label{thm: glw weighted}
Suppose $\Omega$ is a domain in $(M,g)$ with smooth star-shaped boundary.
Assume that the Ricci curvature $\mathrm{Ric}_N$ of $g_N$ satisfies
$\mathrm{Ric}_N\ge (m-1)K g$ and $0\le s'^2-ss''\le K$, where $K>0$ is constant. Suppose $a(r)$ is a positive continuous function such that $b( v^{-1}(u))s( v^{-1}(u))^m$ is convex, where $b(r):=a(r)-a(0)$.
Then the weighted isoperimetric inequality holds: $$\int_{\partial \Omega}a(r)dS\ge \int_{\partial \Omega^\#}a(r)dS.$$
\begin{enumerate}
\item
$(M,g)$ has constant curvature, $a(r)$ is constant on $\partial \Omega$ and $\partial \Omega$ is a geodesic hypersphere, or
\item
$\partial \Omega$ is a slice $\{r=r_0\}$.
\end{enumerate}
\end{theorem}
\begin{proof}
The inequality follows from Theorem \ref{thm: GLW} (\cite[Theorem 1.2]{GLW}) and Theorem \ref{thm: main} (see also Remark \ref{rmk: glw} \eqref{item: mono}) .

Suppose the equality holds, we have two cases: (i) $a(r)|_{\partial \Omega}\not\equiv a(0)$ and (ii) $a(r)|_{\partial \Omega}\equiv a(0)$.\\
\noindent
Case (i).
Let $p\in \partial \Omega$ such that $a(r(p))\ne a(0)$ and $r_0=r(p)$. Then $S=\{q\in \partial \Omega: r(q)=r_0\}$ is clearly closed in $\partial \Omega$. It is also open in $\partial \Omega$ because by \eqref{ineq: I}, $r$ is locally constant on $\{q\in \partial \Omega: a(r(q))\ne a(0)\}$. Therefore $S=\partial \Omega$ and so $\partial \Omega$ is the slice $\{r=r_0\}$. \\
\noindent
Case (ii). In this case, we can without loss of generality assume $a\equiv 1$ on $\partial \Omega$. The equality asserts that $\partial \Omega$ is a smooth hypersurface which has minimum area among all graphical hypersurfaces bounding the same volume, and so by the first variation formulas (e.g. \cite[p. 1186]{osserman1978isoperimetric}), if $\Sigma_t$ is a variation of $\Sigma$ with normal variation $u \nu_{\Sigma_t}$, then
\begin{align*}
\left.\frac{d}{dt}\right|_{t=0}\mathrm{Area}(\Sigma_t)
= m \int_{\Sigma} u H_1 dS=0
\end{align*}
for all $u$ such that
$$\left.\frac{d}{dt}\right|_{t=0}\mathrm{Vol}(\Omega_t)
= \int_{\Sigma} u dS=0.$$
This implies that $\Sigma$ has constant mean curvature . It follows from \cite[Corollary 7]{montiel1999unicity} that either $(M,g)$ has constant curvature and $\partial \Omega$ is a geodesic hypersphere, or $\partial \Omega$ is a slice $\{r=r_0\}$.
\end{proof}

\begin{remark}\label{rmk: glw}
\begin{enumerate}
\item\label{item: mono}
The monotonicity of $s(r)$ is not assumed in Theorem \ref{thm: glw weighted} because of the same reason as Theorem \ref{thm star}. Also, from the proof of Theorem \ref{thm: main}, we only need the classical isoperimetric inequality to hold for star-shaped domains for Theorem \ref{thm: glw weighted} to hold.
\item
By direction computation,
\begin{align*}
&\frac{d^2}{du^2}\left(b( v^{-1}(u)) \;s (v^{-1}(u))^m\right)\\
=&\frac{1}{s(r)^{m+2}}\left(s(r)^2b''(r)+ms(r)s'(r)b'(r)+mb(r)\left(s(r)s''(r)-s'(r)^2\right)\right),
\end{align*}
where $r=v^{-1}(u)$. So the convexity of $b( v^{-1}(u)) \;s( v^{-1}(u))^m$ can be rephrased as
\begin{equation}\label{eq: second der}
s(r)^2b''(r)+ms(r)s'(r)b'(r)-mb(r)\left(s'(r)^2-s(r)s''(r)\right)\ge 0.
\end{equation}
This condition is often easier to check as $v^{-1}$ is usually not very explicit.
\end{enumerate}
\end{remark}
\section{Isoperimetric inequalities involving weighted volume}\label{sec: weighted vol}
In this section, we consider a variant of Theorem \ref{thm: main} involving a weighted volume. In particular, we prove an isoperimetric result without assuming the classical isoperimetric inequality to hold on $M$.

We consider the weighted volume defined by
$$ \mathrm{Vol}_c(\Omega):=\int_{\Omega}c(r)dv_g$$
where $dv_g$ is the $(m+1)$-dimensional volume form with respect to $g$ and $c(r)> 0$ is a radially symmetric weight function. Obviously, this is just the ordinary volume if $c\equiv 1$. We define $\widetilde \Omega^\#$ to be the region $B_R$ which has the same weighted volume as $\Omega$, i.e. $\mathrm{Vol}_c(\widetilde \Omega^\#)=\mathrm{Vol}_c(\Omega)$. Our goal is to look for conditions such that
\begin{align*}
\int_{\partial \Omega}a(r)dS
\ge\int_{\partial \widetilde \Omega^\#}a(r)dS.
\end{align*}

We define the functions $A(r)$, $\widetilde v(r)$ and $\widetilde V(r)$ by
\begin{align*}
A(r):= \prod_{q=1}^{p}s_q(r)^{m_q},\;
\widetilde v(r):= \int_{0}^{r}c(t) A(t) dt\quad \textrm{ and }\;
\widetilde V(r):= \int_{B_r}c \,dv_g =|N|\widetilde v(r).
\end{align*}

\begin{theorem}\label{thm weighted vol}
Let $\Omega$ be a bounded open set in $(M,g)$ with Lipschitz boundary.
Assume that
\begin{enumerate}
\item\label{cond1'}
The projection map $\pi: \partial \Omega\to N$ defined by $(r, \theta)\mapsto \theta$ is surjective.
\item\label{cond2'}
$a(r)$ is a positive continuous function such that
$\psi(r):=a(r)A(r)$ is non-decreasing,
\item \label{cond3'}
The function
$\psi\circ \widetilde V^{-1}$ is convex.
\end{enumerate}
Then
\begin{align*}
\int_{\partial \Omega} a(r)dS
\ge\int_{\partial \widetilde\Omega^{\#}} a(r)dS.
\end{align*}
The equality holds if and only if $r=\mathrm{constant}$, i.e. $\partial \Omega$ is a coordinate slice.
\end{theorem}

\begin{proof}
We use the same notations as the proof of Theorem \ref{thm: main}. Since the proof is similar, we only indicate where changes are made.

We only prove the case where $\Sigma=\partial \Omega$ is piecewise $C^1$, and that $\Sigma$ is a union of graphs over finitely many domains on $N$. So we define $r_{i,j}$ as before. Note that $N=\bigcup_{i=1}^l \overline {S_i}$ by Condition \eqref{cond1'}. Let
\begin{equation*}
I:=\int_{\partial \Omega} a(r)dS\textrm{ and }I^{\#}:=\int_{\partial \widetilde\Omega^{\#}} a(r)dS.
\end{equation*}
As in \eqref{ineq: I},
\begin{equation*}
\begin{split}
I
=&\sum_{i=1}^{l}\sum_{j=1}^{2k_i}\int_{S_i}a(r_{i,j})\prod_{q=1}^{p}\left(1+s_q(r_{i,j})^{-2}\left|\nabla_{N_q}\, r_{i,j}\right|^2_{g_{N_q}}\right)^{\frac{1}{2}}s_q(r_{i,j})^{m_q}d\mathrm{vol}_N\\
\ge&\sum_{i=1}^{l}\int_{S_i} \psi(r_{i, 2k_i}) d\mathrm{vol}_N\\
=&\int_{N}\psi(\rho) d\mathrm{vol}_N
\end{split}
\end{equation*}
where $\rho: N\to [0, \infty)$ is defined by
$\rho(\theta):= \max\{r: (r, \theta)\in \Sigma\} $.

On the other hand, for $B_R=\widetilde\Omega^{\#}$, we have
\begin{align*}
\mathrm{Vol}_c (B_R) = |N| \int_{0}^{R} c(r)A(r)dr=|N|\widetilde v(R).
\end{align*}
As in \eqref{eq: R}, define $R_1$ by
\begin{align*}
\mathrm{Vol}_c (B_R)
= \mathrm{Vol}_c(\Omega ) =&\sum_{i=1}^{l}\sum_{j=1}^{2k_i}(-1)^j\int_{S_i} \widetilde v(r_{i,j}) d\mathrm{vol}_N\\
\le&\sum_{i=1}^{l}\int_{S_i} \widetilde v(r_{i,2k_i}) d\mathrm{vol}_N=: \mathrm{Vol}_c(B_{R_1}) .
\end{align*}
Then analogous to \eqref{eq: R1} and \eqref{ineq: I sharp}, we have
\begin{equation}\label{ineq: I sharp'}
R_1 = \widetilde V^{-1} \left(\int_{N} \widetilde v(\rho) d\mathrm{vol}_N\right)
\; \textrm{and}\;
I^{\#}=|N|\psi(R)\le |N| \psi(R_1).
\end{equation}

As $\psi\circ \widetilde V^{-1}$ is convex,
it is clear that we can proceed as in \eqref{ineq: pf1} to show that $I\ge I^\#$.

The analysis of the equality case and the general case is proved similarly as in Theorem \ref{thm: main}.
\end{proof}

Theorem \ref{thm1} immediately follows from Theorem \ref{thm weighted vol}. For Theorem \ref{thm2}, note that the only place where we have used the monotonicity of $\psi$ (or $s$ in the context of Theorem \ref{thm2}) is \eqref{ineq: I sharp'}. But since $\Sigma$ is star-shaped, $R_1=R$ and the monotonicity condition is not needed.

Let us state only the following version for later use.
\begin{theorem}\label{thm star}
Let $\Omega$ be a bounded open set in $(M,g)$ with Lipschitz star-shaped boundary. Suppose $a(r)$ is positive such that
the function
$\psi\circ \widetilde V^{-1}$ is convex, where $\psi(r)=a(r)A(r)$.
Then
\begin{align*}
\int_{\partial \Omega} a(r)dS
\ge\int_{\partial \widetilde\Omega^{\#}} a(r)dS.
\end{align*}
The equality holds if and only if $r=\mathrm{constant}$, i.e. $\partial \Omega$ is a coordinate slice.
\end{theorem}

\section{Some concrete examples}\label{sec: eg}
In this section, we provide some concrete examples of how Theorem \ref{thm3} can be used to obtain some interesting geometric inequalities.

In all the examples below, the metric $g$ on $M$ is all given by $g=dr^2+s(r)^2 g_{\mathbb S^{m}}$, and the convexity of $b(V^{-1}(u)) s(V^{-1}(u))^m$ is directly checked by using \eqref{eq: second der}. The computations have all been verified by Mathematica.
\begin{enumerate}
\item
On the Euclidean space $\mathbb R^{n}$, the warping function is $s(r)= r$. Choosing
$a(r)=b(r)= r^k $, we have
\begin{equation}\label{ineq: rk}
\int_{\partial \Omega} r^k \,dS\ge \int_{\partial \Omega^\#} r^k \,dS
\end{equation}
if $k\ge 1$. When $k=0$, this is just the classical isoperimetric inequality.
\item
On the hyperbolic space $\mathbb H^n$, the warping function is $s(r)=\sinh r$.
Choosing $a(r)=b(r)=\sinh^k (r)$, we have
\begin{equation*}
\int_{\partial \Omega} \sinh^k r \,dS\ge \int_{\partial \Omega^\#} \sinh^k r \,dS
\end{equation*}
if $k\ge 1$.
Similarly we also have
$$\int_{\partial \Omega} \cosh r \,dS\ge \int_{\partial \Omega^\#} \cosh r \,dS$$
and
$$\int_{\partial \Omega} (\cosh r-1)^k \,dS\ge \int_{\partial \Omega^\#} (\cosh r-1)^k \,dS$$
if $k\ge 1$.
\item
On the open hemisphere $\mathbb S^{n}_+$, the warping function is $s(r)=\sin r$, $(0<r<\frac{\pi}{2})$.
Choosing $a(r)=b(r)=\tan^k (r)$, we have
$$\int_{\partial \Omega} \tan^k r \,dS\ge \int_{\partial \Omega^\#} \tan^k r \,dS.$$
if $k\ge 1$ and $\Omega\subset \mathbb S^n_+$.
Similarly we also have
$$\int_{\partial \Omega} (1-\cos r) \,dS\ge \int_{\partial \Omega^\#} (1-\cos r) \,dS$$
if $k\ge 1$.
\end{enumerate}
In all the above examples, we can convert the inequalities into a form which involves the volume of $\Omega$:
\begin{align*}
\int_{\partial \Omega}a(r)dS\ge
|N|a(R)A(R)
\end{align*}
where $R=V^{-1}(|\Omega|)$. For example, in $\mathbb R^n$, the inequality
$$\int_{\partial \Omega}r^k dS\ge \int_{\partial \Omega^\#}r^k dS$$
is equivalent to
\begin{equation}\label{ineq: Rn}
\int_{\partial \Omega}r^k dS\ge n \beta_n ^{-\frac{k-1}{n}} \mathrm{Vol}(\Omega) ^{\frac{n-1+k}{n}}
\end{equation}
for $k\ge 0$, where $\beta_n$ is volume of the unit ball $B$ in $\mathbb R^n$.
For other spaces, the inequality is not as explicit because $V^{-1}$ is not explicit except when $n=2$.

\section{Weighted isoperimetric theorems involving higher order mean curvatures}\label{sec: higher}
In this section, we generalize the weighted isoperimetric inequality in a warped product manifold to some inequalities involving the weighted integrals of the higher order mean curvatures. This is closely related to the quermassintegral inequalities (\cite{Guan-Li}), which include as a special case the isoperimetric inequality, since the area integral can be interpreted as the integral of the zeroth mean curvature $H_0=1$.

From now on, our warped product manifold is $M^{n} = [0,\lambda)\times {N}^{n-1}$ equipped with the metric $g=dr^2+s(r)^2 g_N$.

Before stating the main theorem, we give some definitions which are useful in studying the extrinsic geometry of hypersurfaces.
On a hypersurface $\Sigma$ in $M$, we define the normalized $k$-th mean curvature function
\begin{eqnarray}
H_k:=H_k(\Lambda)=\frac{1}{\binom{n-1}{k}}\sigma_k(\Lambda),
\end{eqnarray}
where $\Lambda=(\l_1,\cdots,\l_{n-1})$ are the principal curvature functions on $\Sigma$ and the homogenous polynomial $\sigma_k$ of degree $k$ is the $k$-th elementary symmetric function
\[
\sigma_k(\Lambda)=\sum_{i_1<\cdots<i_{k}}\lambda_{i_1}\cdots\lambda_{i_k}.
\]
We adopt the usual convention $\sigma_0=H_{0}=1$.

The $k$-th Newton transformation $T_k: T\Sigma \rightarrow T \Sigma$ (cf. \cite{reilly1973variational}) is useful in studying the extrinsic geometry of $\Sigma$, and is defined as follows.
If we write
\begin{equation*}
T_k ( e_j ) =\sum_{i=1}^{n-1} ( T_k )_j^i e_{i},
\end{equation*}
then $ (T_k) _j^i $ are given by
$$
{(T_k)}_j^{\,i}= \frac 1 {k!}
\sum_{\substack{1 \le i_1,\cdots, i_k \le n-1\\ 1\le j_1, \cdots, j_k \le n-1}}
\delta^{i i_1 \ldots i_k}_{j j_1 \ldots j_k}
B_{i_1}^{j_1}\cdots B_{i_k}^{j_k}
$$
where $B$ is the second fundamental form of $\Sigma$.
One also defines $T_0 = \mathrm{Id}$, the identity map.

We define the vector field $X = s(r) \, \frac{\partial}{\partial r}$ and the potential function $c(r) = s'(r)$.
Note that $X$ is a conformal Killing vector field: $\mathcal L_X g = 2 c g$ \cite[Lemma 2.2]{B2013}.

The warped product manifold is somewhat special in that there exists a nontrivial conformal Killing vector field, which in turn leads to some nice formulas of Hsiung-Minkowski types. We will need the following weighted Hsiung-Minkowski formulas (cf. \cite[Proposition 2.1]{Kwo2016}, \cite[Proposition 1]{KLP}):
\begin{proposition}[Weighted Hsiung-Minkowski formulas]\label{prop: HM}
Suppose $\eta$ is a smooth function on a closed hypersurface $\Sigma$ in $M$, then for $1\le k \le n-1$, we have
\begin{equation*} \label{weighted in}
\begin{split}
\int_\Sigma \eta c H_{k-1}dS =&\int_\Sigma \eta H_{k} \langle X, \nu\rangle dS
-\frac{1}{k{{n-1}\choose k}}\int_\Sigma \eta \left(\mathrm{div}_\Sigma T_{k-1}\right)(X^T)dS\\
&-\frac{1}{k{{n-1}\choose k}}\int_\Sigma \langle T_{k-1}(X^T), \nabla _\Sigma \eta\rangle dS,
\end{split}
\end{equation*}
where $\nu$ is the unit normal vector, $X=s (r) \partial _r$ and $X^T$ is the tangential component of $X$ onto $T\Sigma$.
\end{proposition}
\begin{proof}
For completeness we sketch the proof here. Let $m=n-1$. We compute
\begin{align*}
&\mathrm{div}_\Sigma \left(\eta T_k (X^T)\right)\\
=&\langle T_k(\nabla \eta), X^T\rangle+ \eta(\mathrm{div}\;T_k)(X^T)+ \frac{1}{2}\eta\langle T_k, \iota^* (\mathcal{L}_Xg )\rangle -\eta\langle T_k, B\rangle \langle X, \nu\rangle\\
=&\langle T_k(\nabla \eta), X^T\rangle+ \eta(\mathrm{div}\;T_k)(X^T)+ c \eta\langle T_k, \iota^*g  \rangle -\eta\langle T_k, B\rangle \langle X, \nu\rangle\\
=&\langle T_k(\nabla \eta), X^T\rangle+ \eta(\mathrm{div}\;T_k)(X^T)+ (m-k){m\choose k}c \eta H_k -(k+1){m\choose {k+1}}H_{k+1}\eta \langle X, \nu\rangle
\end{align*}
where $\iota$ is the inclusion of $\Sigma$ in $M$ and we used the fact that $\mathrm{tr}_\Sigma(T_k)= (m-k){m\choose k}H_k$ and $\langle T_k, B\rangle =(k+1){m\choose {k+1}} H_{k+1}$. Applying the divergence theorem will then give the result.
\end{proof}

\begin{lemma} \label{T positive}
Suppose $N$ has constant curvature $K$ and $\Sigma$ is a star-shaped hypersurface with $H_p>0$.
Assume that $s'>0$ for $r>0$ and $s'(r)^2-s(r)s''(r)\le K$.
Then
\begin{enumerate}
\item\label{item: 1}
For all $k \in \{1, \cdots, p-1\}$, we have
$T_k>0$ and $H_k>0$.
\item\label{item: 2}
For $k \in \{2, \cdots, p\}$,
\begin{equation}\label{eq: div}
(\mathrm{div}_\Sigma T_{k-1}) (X^T)\ge 0,
\end{equation}
where $X^T$ is the tangential component of $X$ onto $T\Sigma$.
\end{enumerate}
\end{lemma}

\begin{proof}

This is essentially \cite[Lemma 1, Proposition 1]{KLP} or \cite[Section 2]{BE2013}, despite some minor differences in the assumptions. \eqref{item: 1} is proved in \cite[Lemma 1 (2b)]{KLP}. \eqref{item: 2} follows the same proof as in \cite[Proposition 1 (2)]{KLP}. In the proof of \cite[Proposition 1 (2)]{KLP}, it is assumed that $K>s'(r)^2-s(r)s''(r)$ (Condition (H4) in \cite{BE2013}, \cite{B2013}) and $\langle X, \nu\rangle >0$, but since we only require non-strict inequality in \eqref{eq: div}, the conclusion still holds under our assumption.

A remark is that we need $N$ to have constant curvature because conformal flatness of $g$ is essential in the formula of $\left(\mathrm{div}_{\Sigma}T_{k-1}\right)(X^T)$ on p. 393 in \cite{BE2013}.

\end{proof}

\begin{theorem}\label{thm: mean curvature}
Suppose $\Omega$ is a domain in $M$ and its boundary
is a smooth hypersurface with $H_1 \ge 0$.
Assume that $s'>0$ for $r>0$ and
$0\le s'(r)^2-s(r)s''(r)$.
Then for $l\ge 1$,
\begin{align*}
|N|^{-\frac{l-1}{n}} \left(n\int_{\Omega}c(r) dv\right)^{\frac{n+l-1}{n}}
\le \int_{\partial \Omega} H_1 s(r)^{l+1} c(r)^{-1} \,dS.
\end{align*}
The equality holds if and only if $\partial \Omega$ is a slice.
\end{theorem}

\begin{proof}
We will prove the following chain of inequalities:
\begin{equation*}
\begin{split}
|N|^{-\frac{l-1}{n}} \left(n\int_{\Omega}c(r) dv\right)^{\frac{n+l-1}{n}}
\le \int_{\partial \Omega} s(r)^{l} dS
\le \int_{\partial \Omega} H_1 s(r)^{l+1} {c(r)}^{-1} \,dS.
\end{split}
\end{equation*}
Define $a(r)=s(r)^l $.
As $A(r)=s(r)^{n-1}$, $\psi(r)=s(r)^{n-1+l} $ and $\widetilde v(r)=\int_{0}^{r}s(t)^{n-1} c (t) dt=\frac{1}{n} s(r)^n$, we have
\begin{align*}
\psi\circ \widetilde v^{-1}(u)
=& \left( n u\right)^{\frac{n-1+l}{n}}
\end{align*}
which is clearly convex as $l\ge 1$. So by Theorem \ref{thm star},
we have
\begin{equation}\label{eq: star1}
\int_{\partial \Omega}s(r)^l dS\ge \int_{\partial \widetilde \Omega^\#}s(r) ^l dS
=|N|^{-\frac{l-1}{n}} \left(n\int_{\Omega}c(r) dv\right)^{\frac{n+l-1}{n}}.
\end{equation}

We now simply denote $s(r)$ by $s$ and $c(r)$ by $c$.
Applying the weighted Hsiung-Minkowski formula (Proposition \ref{prop: HM}), we have
\begin{align*}
\int_{\partial \Omega} s^l dS
=\int_{\partial \Omega} c\frac{s^l}{c} dS
=&\int_{\partial \Omega} \frac{s^l}{c} H_1 \langle X, \nu \rangle-\frac{1}{n-1} \int_{\partial \Omega} \left\langle \nabla \left(\frac{s^l}{c}\right), s \nabla r\right\rangle dS\\
=&\int_{\partial \Omega} \frac{s^l}{c} H_1 \langle X, \nu \rangle-\frac{1}{n-1} \int_{\partial \Omega} \left\langle \frac{s^{l-1} }{c^2}(lc^2-ss'')\nabla r, s \nabla r\right\rangle dS\\
\le&\int_{\partial \Omega} \frac{s^l}{c} H_1 \langle X, \nu \rangle dS\\
\le&\int_{\partial \Omega} \frac{s^{l+1} }{c} H_1 dS.
\end{align*}
\end{proof}

To relate the weighted volume to the integral of the higher order mean curvatures, we need stronger assumptions.
\begin{theorem}\label{thm: higher}
Suppose $N$ has constant curvature $K$ and $\Sigma$ is a closed star-shaped hypersurface which is the boundary of a domain $\Omega$.
Assume that $s'>0$ for $r>0$ and
$0\le s'(r)^2-s(r)s''(r)\le K$ and $H_k>0$ on $\Sigma$.
Then for
$l\ge 1$,
\begin{align*}
|N|^{-\frac{l-1}{n}} \left(n\int_{\Omega}c(r) \, dv\right)^{\frac{n+l-1}{n}}
\le \int_{\partial \Omega} H_k s(r)^{l+k} c(r)^{-k} \,dS.
\end{align*}
The equality holds if and only if $\Sigma$ is a slice.
\end{theorem}

\begin{proof}
We actually prove the following stronger statement:
\begin{equation}\label{ineq: chain}
\begin{split}
|N|^{-\frac{l-1}{n}} \left(n\int_{\Omega}c(r) dv\right)^{\frac{n+l-1}{n}}
\le \int_{\partial \Omega} s(r)^{l} dS
\le& \int_{\partial \Omega} H_1 s(r)^{l+1} c(r)^{-1} dS\\
\le& \cdots\\
\le& \int_{\partial \Omega} H_k s(r)^{l+k} c(r)^{-k} dS.
\end{split}
\end{equation}

The first two inequalities have already been proved in Theorem \ref{thm: mean curvature}.

We prove the remaining inequalities by induction. Again denote $s(r)$
by $s$ and $c(r)$ by $c$. By the weighted Hsiung-Minkowski formula (Proposition \ref{prop: HM}) and Lemma \ref{T positive}, for $1\le j\le k-1$, we have
\begin{align*}
&\int_{\partial \Omega} \frac{s^{l+j}}{c ^j }H_j dS\\
=&\int_{\partial \Omega} \frac{s^{l+j} }{c^{j+1} } H_{j+1} \langle X, \nu \rangle dS
-\frac{1}{j{{n-1}\choose j}} \int_{\partial \Omega} \left\langle \frac{s^{l+j-1} }{c ^{j+2} }((l+j)c^2 -(j+1)ss'' )T_j(\nabla r), s  \nabla r\right\rangle dS\\
\le&\int_{\partial \Omega} \frac{s^{l+j} }{c^{j+1} } H_{j+1} \langle X, \nu \rangle dS\\
\le&\int_{\partial \Omega} \frac{s^{l+j+1} }{c^{j+1} } H_{j+1}dS.
\end{align*}

\end{proof}
We note that the conditions in Theorem \ref{thm: mean curvature} and Theorem \ref{thm: higher} are satisfied in the following space forms:
\begin{enumerate}
\item
The Euclidean space $\mathbb R^n$ with metric $dr^2+r^2 g_{\mathbb S^{n-1}}$.
\item
The hyperbolic space $\mathbb H^n$ with metric $dr^2+\sinh^2 r g_{\mathbb S^{n-1}}$.
\item
The open hemisphere $\mathbb S^n_+$ with metric $dr^2+\sin^2 r g_{\mathbb S^{n-1}}$.
\end{enumerate}
In the following corollaries, we denote the point $\{r=0\}$ by $0$ and the volume of the unit ball in $\mathbb R^n$ by $\beta_n$. We obtain the following corollaries.
\begin{corollary}\label{cor: Rn}
Let $\Sigma$ be a closed embedded hypersurface in $\mathbb R^{m+1}$
which is star-shaped with respect to $0$ and $\Omega$ is the region enclosed by it. Assume that $H_k>0$ on $\Sigma$.
Then for any integer $l\ge 0$,
\begin{align*}
n \beta_n^{-\frac{l-1}{n}}\mathrm{Vol}(\Omega)^{\frac{n-1+l}{n}}
\le\int_{\Sigma} H_kr^{l+k}dS.
\end{align*}
\end{corollary}
\begin{proof}
The case where $l\ge 1$ follows directly from Theorem \ref{thm: higher}. If $k=l=0$, this is the ordinary isoperimetric inequality. So \eqref{eq: star1} is still true when $l=0$, and we can perform induction starting from this case to show the assertion when $l= 0$.
\end{proof}

\begin{remark}
If $l=1$, the inequality becomes
\begin{align*}
n\mathrm{Vol} (\Omega) \le \int _{\partial \Omega} H_kr^{k+1}dS.
\end{align*}
In particular, if $k=1$, it is easily seen from the proof that the assumption can be weakened to $H_1\ge 0$ because $T_0=\mathrm{id}$ is always positive. Corollary \ref{cor: Rn} extends \eqref{ineq: Rn} in Section \ref{sec: eg} and also generalizes \cite{KM2014} Theorem 2. See also \cite[Theorem 2]{kwong2015monotone}.
\end{remark}

\begin{corollary}
Let $\Sigma$ be a closed embedded hypersurface in $\mathbb H^{m+1}$
which is star-shaped with respect to $0$ and $\Omega$ is the region enclosed by it. Assume that $H_k>0$ on $\Sigma$.
Then for any integer $l\ge 1$,
\begin{align*}
n \beta_n^{-\frac{l-1}{n}} \left(\int_{\Omega}\cosh r \,dv\right)^{\frac{n+l-1}{n}}
\le \int_{\partial \Omega} H_k \sinh^l r\tanh ^k r \,dS.
\end{align*}
\end{corollary}

\begin{corollary}
Let $\Sigma$ be a closed embedded hypersurface in $\mathbb S_+^{m+1}$
which is star-shaped with respect to $0$ and $\Omega$ is the region enclosed by it. Assume that $H_k>0$ on $\Sigma$.
Then for any integer $l\ge 1$,
\begin{align*}
n \beta_n^{-\frac{l-1}{n}} \left(\int_{\Omega}\cos r \, dv\right)^{\frac{n+l-1}{n}}
\le \int_{\partial \Omega} H_k \sin^l r\tan ^k r \,dS.
\end{align*}
\end{corollary}

It is also possible to prove results analogous to Theorem \ref{thm: higher} for standard space forms by extending the inequalities in Section \ref{sec: eg} using the weighted Hsiung Minkowski inequalities, we will not do it here for the sake of simplicity.

\section{Applictions to eigenvalue estimates}\label{sec: eigen}
In this section, we apply Theorem \ref{thm3} to obtain some sharp eigenvalue estimates.

We define $\lambda_1(T_k)$ to be the first eigenvalue of the symmetric second order differential operator $-\mathrm{div}(T_k\circ\nabla )$ on $\Sigma$. The equality holds if and only if $\Sigma$ is immersed as a geodesic sphere. Note that $\lambda_1(T_0)$ is just the first Laplacian eigenvalue.

We now give an application of our main result to eigenvalues estimatation. The following theorem generalizes \cite{wang2010isoperimetric} Theorem 1.2, which corresponds to the case where $k=0$.
\begin{theorem}\label{thm: lambda}
\label{thm: lambda1}
Let $\Sigma$ be a closed embedded hypersurface in $\mathbb R^{m+1}$ enclosing a region $\Omega$.
Then
\begin{align*}
\lambda_1(T_k)\le \frac{(m-k){m\choose k}\beta^{\frac{1}{n}}}{n \mathrm{Vol}(\Omega)^{\frac{n+1}{n}}}\int_{\partial \Omega} H_k dS
\end{align*}
where $n=m+1$ and $\beta_n$ is the volume of the unit ball in $\mathbb R^n$.
The equality holds if and only if $\Sigma$ is a sphere. (Note that $\lambda_1(T_0)$ is just the first Laplacian eigenvalue.)
\end{theorem}

\begin{proof}
By a suitable translation, we can assume that $\int_{\partial \Omega}x_i\,dS=0$ for $i=1, \cdots, n$.

By Theorem \ref{thm3}, we have
\begin{equation}\label{ineq: volr}
\int_{\partial \Omega} r^2 dS
\ge \int_{\partial \Omega^\#}r^2 dS
= n \beta_n^{-\frac{1}{n}} \mathrm{Vol}(\Omega)^{\frac{n+1}{n}}.
\end{equation}

By the variational characterization of $\lambda_1(T_k)$ and the fact that $\textrm{tr}_\Sigma(T_k)=(m-k)H_k$, we have
\begin{equation}\label{ineq: lambda}
\begin{split}
\lambda_1(T_k) \int_{\partial \Omega} r^2 =&\lambda_1(T_k)\sum_{i=1}^n \int_{\partial \Omega} {x_i} ^2
\le \int_{\partial \Omega} \sum_{i=1}^n \langle T_k (\nabla x_i), \nabla x_i\rangle dS\\
=&\int_{\partial \Omega}\sum_{j,l=1}^m \left(\sum_{i=1}^n(\nabla _{e_j}x_i)(\nabla _{e_l}x_i) \right)(T_k)_j^l dS\\
=&\int_{\partial \Omega} \mathrm{tr}_\Sigma(T_k) dS\\
=&\int_{\partial \Omega} (m-k){m\choose k}H_k dS.
\end{split}
\end{equation}
Therefore combining the two inequalities we have
\begin{align*}
n \beta_n^{-\frac{1}{n}} \mathrm{Vol}(\Omega)^{\frac{n+1}{n}} \lambda_1(T_k)
\le (m-k){m\choose k}\int_\Sigma H_k dS.
\end{align*}
If the equality holds, then by Theorem \ref{thm3}, $\partial \Omega$ is a sphere.
\end{proof}

\begin{corollary}
Let $\Sigma$ be a closed embedded hypersurface in $\mathbb R^{m+1}$ enclosing a region $\Omega$. Then the first Laplacian eigenvalue $\lambda_1(\Sigma)$ on $\Sigma$ satisfies
\begin{align*}
\lambda_1(\Sigma)\le \frac{m\beta_n^{\frac{1}{n}} \mathrm{Area}( \Sigma)}{n\mathrm{Vol}(\Omega)^{\frac{n+1}{n}}}.
\end{align*}
The equality is attained if and only if $\Omega$ is a ball.
\end{corollary}
To state our next result, we need to define the Steklov eigenvalues, as follows. Let $(\Omega,g)$ be a compact Riemannian manifold with smooth boundary $\partial \Omega=\Sigma$. The first nonzero Steklov eigenvalue is defined as the smallest $p\ne0$ of the following Steklov problem (\cite{Stekloff})
\begin{equation}\label{eq: stek}
\begin{cases}
\Delta f =0\quad &\textrm{on }\Omega\\
\frac{\partial f}{\partial \nu}=p f \quad &\textrm{on }\partial \Omega
\end{cases}
\end{equation}
where $\nu$ is the unit outward normal of $\partial \Omega$.
Physically, this describes the stationary heat distribution in a body $\Omega$ whose flux through $\partial \Omega$ is proportional to the temperature on $\partial \Omega$. It is known that the Steklov boundary problem \eqref{eq: stek} has a discrete spectrum
$$0=p_0< p_1\le p_2\le\cdots \to \infty.$$
Moreover, $p_1$ has the following variational characterization
(e.g. \cite[Equation 2.3]{KS})
\begin{equation}\label{eq: p1}
p_1(\Omega)=\min_{\int_{\partial \Omega}f dS=0} \frac{\int_\Omega|\nabla f|^2 dv}{\int_{\partial \Omega}f^2 dS}.
\end{equation}

We will now prove an upper bound of $p_1(\Omega)$ with the techniques similar to that in Theorem \ref{thm: lambda}.
\begin{theorem}
For a domain $\Omega$ in $\mathbb R^n$ with smooth boundary, the first Steklov eigenvalue $p_1$ satisfies
\begin{equation*}
p_1(\Omega)\le \left(\frac{\beta_n}{\mathrm{Vol}(\Omega)}\right)^{\frac{1}{n}}.
\end{equation*}
The equality is attained if and only if $\Omega$ is a ball.
\end{theorem}

\begin{proof}
As in the proof of Theorem \ref{thm: lambda}, we assume $\int_{\partial \Omega}x_i dS=0$ for $i=1, \cdots, n$. So by \eqref{ineq: volr} and  \eqref{eq: p1},
\begin{align*}
n \beta_n^{-\frac{1}{n}} \mathrm{Vol}(\Omega)^{\frac{n+1}{n}}
\le \int_{\partial \Omega} r^2 dS
=& \sum_{i=1}^{n}\int_{\partial \Omega} x_i^2 dS\\
\le& \sum_{i=1}^{n}\frac{1}{p_1} \int_{\Omega}|\nabla x_i^2|dv\\
=&\frac{1}{p_1} n\mathrm{Vol}(\Omega).
\end{align*}
From this we obtain
\begin{equation*}
p_1\le \left(\frac{\beta_n}{\mathrm{Vol}(\Omega)}\right)^{\frac{1}{n}}.
\end{equation*}
If the equality holds, then by Theorem \ref{thm3}, $\partial \Omega$ is a sphere.
\end{proof}

\section{The necessity of the conditions}\label{sec: necess}
In this section, we examine the necessity of the conditions in Theorem \ref{thm1}, the classical isoperimetric inequality.

First, we consider the condition that $A\circ V^{-1}$ being a convex function (Assumption \ref{cond: convex}), or equivalently that $ss''-s'^2\ge 0$ if $s$ is twice differentiable. We will show that this condition is necessary in the following sense:
\begin{proposition}\label{prop: necess}
Suppose $s(r_0)s''(r_0)-s'(r_0)^2<0$, then there exists a compact Riemannian manifold $(N, g_N)$ such that for the warped product manifold $([0, \lambda)\times N, dr^2+s(r)^2g_N)$, the coordinate slice $\Sigma=\{r=r_0\}$ fails to be area-minimizing among nearby hypersurfaces enclosing the same volume.
\end{proposition}
\begin{proof}
The second variation formula for area on a constant-mean-curvature hypersurface $\Sigma$, which is a critical point of the area functional subject to the constraint that a fixed amount of volume is enclosed, reads (e.g. \cite[Equation (1)]{CY})
\begin{align*}
\left.\frac{d^2}{dt^2}\right|_{t=0}\mathrm{Area}(\Sigma_t)
= \int_{\Sigma} \left(|\nabla_\Sigma u|^2-\left(|B|^2+\mathrm{Ric}_g (\nu, \nu)\right)u^2\right)dS
\end{align*}
where the 1-parameter family of deformations of $\Sigma$ is given by $\Sigma_t=\phi(\Sigma, t)$ with $\Sigma_0=\Sigma$, $|B|^2$ is the square norm of the second fundamental form, $u$ is a function on $\Sigma$ and the normal variation is $u\nu_{\Sigma_t}$. In order to preserve the volume, we require $\int_\Sigma u dS=0$.

Suppose $s(r_0)s''(r_0)-s'(r_0)^2=-k<0$.
Choose a compact Riemannian manifold $(N, g_N)$ such that its first Laplacian eigenvalue $\lambda_1(g_N)<m k$. Such an $N$ can, for example, be a sphere $\mathbb S^m(R)$ with sufficiently large radius $R$ such that $\frac{1}{R^2}<k$.

Let $M=[0, \lambda)\times \mathbb R$ equipped with the metric $g=dr^2+s(r)^2 g_N$, then its straightforward to compute that
\begin{equation}\label{eq: ric}
|B|^2+\mathrm{Ric}_g(\nu, \nu)=m\left(\frac{s'(r_0)^2-s(r_0)s''(r_0)}{s(r_0)^2}\right)
\end{equation}
on the hypersurface $\Sigma=\{r=r_0\}$.
It is also easy to see that $u$ is an eigenfunction on $\Sigma$ with eigenvalue $\frac{\lambda_1(g_N)}{s(r_0)^2}$, so by the variational characterization of $\lambda_1$, $u$ satisfies $\int_\Sigma u dS=0$ and
\begin{align*}
\int_\Sigma |\nabla_\Sigma u|^2 dS=\frac{\lambda_1(g_N)}{s(r_0)^2}\int_\Sigma u^2 dS.
\end{align*}
Therefore the second variation formula becomes
\begin{align*}
\left.\frac{d^2}{dt^2}\right|_{t=0}\mathrm{Area}(\Sigma_t)
=& \int_{\Sigma} \left(|\nabla_\Sigma u|^2-\left(|B|^2+\mathrm{Ric}_g (\nu, \nu)\right)u^2\right)dS\\
=& \left(\lambda_1(g_N)-m\left(s'(r_0)^2-s(r_0)s''(r_0)\right)\right)\int_{\Sigma} \frac{u^2}{s(r_0)^2} dS\\
<&0.
\end{align*}
It follows that $\Sigma$ fails to be area-minimizing among nearby hypersurfaces enclosing the same volume.
\end{proof}
\begin{remark}
The necessity of the condition $ss''-s'^2\ge -1$ when $N$ is the unit sphere is also discussed in \cite{chunhe2016necessary}. See also \cite{GLW} Remark 6.3. We notice that one of the conditions of Theorem 1.2 (isoperimetric inequality) in \cite{GLW} is that $-k\le ss''-s'^2\le 0$. Proposition \ref{prop: necess} does not contradict the result in \cite{GLW} because it is assumed that $\mathrm{Ric}_N \ge (m-1)k g_N$ in \cite{GLW} while in our example, $\mathrm{Ric}_N<(m-1)k g_N$. Indeed, as already noted in \cite{GLW}, the condition that $\mathrm{Ric}_N \ge (m-1)k g_N$ guarantees that $\lambda_1(g_N)\ge m k$ by Lichnerowicz' theorem \cite{Lich}.
\end{remark}

It is also easy to see that the condition that the surjectivity of the projection map $\pi: \partial \Omega\to N$ is necessary if $s(0)>0$, or equivalently, $A(0)>0$. Indeed, if $A(0)>0$, $|B_r|\to 0$ but $\mathrm{Area}(\{r=r_0\})\to A(0)$ as $r\to 0^+$. If we take $\Omega$ to be a small enough geodesic ball around a point which is far from $r=0$, then the isoperimetric inequality clearly fails as $\partial \Omega$ has smaller area than $\partial B_r$, where $B_r=\Omega^\#$.

We now show that in the case where $s(0)=0$, Assumption\ref{cond: surj} in Theorem \ref{thm1} cannot be removed unless we impose more restriction on $s'(0)$.
\begin{proposition}\label{prop: s'(0)}
Suppose $s(0)=0$.
If $s'(0) >\left(\frac{n \beta_n}{|N|}\right)^{\frac{1}{n-1}}$, then Theorem \ref{thm1} fails if Assumption \ref{cond: surj} is removed. In particular, if $N=\mathbb S^{n-1}$ and $s'(0)> 1$, then Theorem \ref{thm1} fails if Assumption \ref{cond: surj} is removed.
\end{proposition}

\begin{proof}
We now use the notation that $B_R(0)=\{(r, \theta): r<R\}$ and $B_R(p)$ to be the geodesic ball of radius $R$ around a point $p$ in $M$ with $r(p)\ne 0$.
We will compare the areas of $\partial B_r(p)$ and $\partial B_r(0)$ for small $r$.

By Taylor's theorem, $s(r)=s'(0)r+O(r^2)$ as $r\to 0^+$. Then
\begin{align*}
|B_r(0)|=&\frac{|N|}{n} s'(0)^{n-1}r^n+O(r^{n+1})\quad \textrm{ and }\\
|\partial B_r(0)|=&|N| s(r)^{n-1}= |N| s'(0)^{n-1}r^{n-1}+O(r^n).
\end{align*}

On the other hand, if $r(p)\ne 0$, then $|B_r(p)|=\beta_n r^n +O(r^{n+2})$ and $\partial B_r(p)=n \beta_n r^{n-1}+O(r^{n+1})$ (cf. \cite[Theorem 3.1]{gray1974volume}).

Fix a small geodesic ball $B_{r}(p)$. If $|B_{r_0}(0)|=|B_{r}(p)|$, we must have $r_0=\left(\frac{n \beta_n}{|N|}\right)^{\frac{1}{n}}s'(0)^{-\frac{n-1}{n}} r+O({r}^2)$. So
\begin{align*}
|\partial B_{r_0}(0)|
=&|N|s(r_0)^{n-1}\\
=&|N| s'(0)^{n-1}\left(\left(\frac{n \beta_n}{|N|}\right)^{\frac{1}{n}}s'(0)^{-\frac{n-1}{n}} r\right)^{n-1}+O({r}^{n})\\
=&|N|^{\frac{1}{n}}\left(n \beta_n\right)^{\frac{n-1}{n}} s'(0)^{\frac{n-1}{n}} r ^{n-1}+O({r}^{n}).
\end{align*}
Therefore, for the isoperimetric inequality to hold, it is necessary that $|N|^{\frac{1}{n}}\left(n \beta_n\right)^{\frac{n-1}{n}} s'(0)^{\frac{n-1}{n}} \le n \beta_n$, i.e. $s'(0)\le \left(\frac{n \beta_n}{|N|}\right)^{\frac{1}{n-1}}$. In particular, if $N=\mathbb S^{n-1}$, this means that $s'(0)\le 1$.
\end{proof}
In particular, Proposition \ref{prop: s'(0)} implies that for a given $s(r)$ satisfying all the conditions in Theorem \ref{thm1} and such that $s(0)=0$ and $s'(0)>0$, we can choose a Riemannian manifold $N$ with large enough volume $|N|$ such that $s'(0) >\left(\frac{n \beta_n}{|N|}\right)^{\frac{1}{n-1}}$ (for example by choosing a sphere with large enough radius), then the isoperimetric inequality will fail if we drop the Assumption \ref{cond: surj}. Alternatively, we can also fix $N$ and rescale $s(r)$, which does not affect the assumptions in Theorem \ref{thm1}.

Finally, we give an example in which the isoperimetric inequality fails when all assumptions except Assumption \eqref{cond: mono} hold. In view of Theorem \ref{thm2}, the counterexample must not be a star-shaped hypersurface. For convenience, we take the interval to be $[1, \infty)$ and $N$ to be any compact $m$-dimensional manifold. Let $s(r)=r^{-\frac{1}{m}}$ which is a decreasing function. As $\log s(r)=-\frac{1}{m}\log r$ is convex, \eqref{cond: convex} is satisfied. For the region $\Omega=\{R_1<r<R_2\}$, $|\Omega| =|N|(\log R_2 -\log R_1)$ and $ |\partial \Omega| = |N|({R_1}^{-1}+{R_2}^{-1})$. If we take $R_2=e R_1$ and let $R_1\to \infty$, then we have $|\Omega|\equiv |N|$ but $|\partial \Omega|\to 0$. Therefore the isoperimetric inequality fails.

\end{document}